\documentclass[10pt]{article}
\usepackage{amsfonts}

\usepackage{graphicx}
\usepackage{subcaption}

\usepackage[utf8]{inputenc}
\usepackage[top=30mm, left=30mm, right=30mm, bottom=30mm]{geometry}
\usepackage{amsmath,amssymb,amsthm}
\usepackage{dsfont}
\usepackage{color}
\usepackage{lineno,hyperref}
\usepackage{yfonts}
\usepackage{appendix}
\usepackage{babel}

\newtheorem{theorem}{Theorem}[section]
\newtheorem{lem}[theorem]{Lemma}

\newtheorem{prop}[theorem]{Proposition}
\newtheorem{rem}[theorem]{Remark}
\theoremstyle{definition}
\newtheorem{dfn}[theorem]{Definition}
\numberwithin{equation}{section}
\DeclareMathOperator*{\supp}{supp}

\newcommand*{\Bscr}
{\mathcal B}

\newcommand*{\Fscr}
{\mathcal F}

\newcommand*{\R}
{\mathbb{R}}

\definecolor{orange}{rgb}{1.0, 0.55, 0.0}

\newcommand{\footremember}[2]{%
	\footnote{#2}
	\newcounter{#1}
	\setcounter{#1}{\value{footnote}}%
}

\begin{document}
	
\title{Average dissipation for stochastic transport equations with Lévy noise}
	\author{
		Franco Flandoli \footnote{Faculty of Sciences, Scuola Normale Superiore Pisa, Italy. E-mail: franco.flandoli@sns.it}
		\and Andrea Papini \footnote{Department of Mathematical Sciences,
			Chalmers University of Technology \& University of Gothenburg,
			Gothenburg, Sweden. E-mail: andreapa@chalmers.se}
		\and Marco Rehmeier\footremember{alley}{Faculty of Sciences, Scuola Normale Superiore Pisa, Italy} \footnote{Faculty of Mathematics, Bielefeld University, Bielefeld, Germany. E-mail: mrehmeier@math.uni-bielefeld.de}%
	}
	\date{}
	\maketitle

\begin{abstract}
We show that, in one spatial and arbitrary jump dimension, the averaged solution of a Marcus-type SPDE with pure jump Lévy transport noise satisfies a dissipative deterministic equation involving a fractional Laplace-type operator. To this end, we identify the correct associated Lévy measure for the driving noise. We consider this a first step in the direction of a non-local version of enhanced dissipation, a phenomenon recently proven to occur for Brownian transport noise and the associated local parabolic PDE by the first author. Moreover, we present numerical simulations, supporting the fact that dissipation occurs for the averaged solution, with a behavior akin to the diffusion due to a fractional Laplacian, but not in a pathwise sense.
\end{abstract}
\\
\textbf{MSC2020 subject classifications:} 60H15; 76F55; 76F25; 35Q49; 35K10; 60J76

\section{Introduction}

Stochastic transport (and advection, not discussed in this paper) attracts
more and more attention for its potentialities to describe small scale
turbulence in several models and applications, see for instance the volumes 
\cite{FL23,FL23-book,FMP24} or the application to
raindrop formation \cite{FHP23} and turbulence in pipes \cite{CMT23}%
, among many others. Small scale turbulence is described by stochastic
processes, space-dependent, either given a priori or inferred from data.
Most models deal with white noise or Ornstein--Uhlenbeck processes, the first
basic paradigms for any investigation of this kind. However, turbulent
signals may be more complex. Two classes of processes seem to be the first
ones to be considered after Gaussian noise: Fractional Gaussian noise, and $%
\alpha $-stable processes. The first one has been considered in \cite%
{FR23}, reporting some preliminary results. In this paper we give some
preliminary results on $\alpha $-stable processes, which seem to be the first of its kind. We also include a list of questions for future work in this direction.

The property of turbulent fluids we want to emphasize is the additional
dissipation produced by turbulent eddies. In the white noise case this has
been widely investigated, see for instance \cite{FL23-book} and the references
therein. Here, we consider this property in the case of $\alpha $-stable transport
noise. The most interesting case is undoubtedly when this transport noise
models the small scales acting on the large ones in nonlinear models, but
such a case is still too difficult in the $\alpha $-stable case. We limit
ourselves to linear transport of a scalar quantity, for instance heat,
by an $\alpha $-stable noise. 
Even in this simplified setting, many difficult technical questions emerge.
The first one is which notion of integral should be used. In the
Gaussian white noise case the basic rule is to choose a noise providing the
correct invariance (conservation) properties. This leads to Stratonovich
noise. In the $\alpha $-stable case the same invariance is given by the
so-called Markus noise. This type of stochastic integral has already been developed in the
literature, both in the finite- and infinite-dimensional case \cite{M78,CP14,HP23,Applebaum09_Levy-proc}. Its main advantage is that it preserves the ordinary rules of calculus for general Lévy processes as integrators. In this sense, the Marcus integral can be considered a natural extension of the Stratonovich integral. Indeed, for diffusion integrators, both integral notions coincide. We briefly review its definition in Section \ref{sect:M-SPDE}.

Given the model, the noise and the meaning of stochastic calculus, we limit
ourselves here to investigate the following question: whether the expected
value of the solution is dissipated, and whether this expected value satisfies
a closed equation, similarly to the Gaussian noise case. We give affirmative answers to  these questions: the expected value of the solution to the Marcus-type SPDE \eqref{M-SPDE}, for a suitable choice of an essentially $\alpha$-stable symmetric Lévy measure, satisfies the second-order parabolic deterministic equation \eqref{eq:det-eq}, see Proposition \ref{prop}. The precise shape of the operator in \eqref{eq:det-eq} depends, of course, on the chosen Lévy measure for the pure jump Lévy process $Z$ from \eqref{M-SPDE}. With the aforementioned choice of an $\alpha$-stable symmetric Lévy measure, the resulting operator is close (but not identical) to the $\alpha$-fractional Laplacian.

In order to support our claims, we include  numerical simulations, which underline the dissipative character of the expected value of the solution to \eqref{M-SPDE}. We limit ourselves to dimension $d=1$, in which an explicit solution can be computed via the method of characteristics. 
We simulate the evolution of the averaged solution profile, obtained with Monte Carlo method, showing a decay in both time and space, as expected from our theoretical results. More so, starting with a compact supported initial condition, we analyze, in Figures \ref{da_s05}, \ref{regsol_s05} and \ref{log_s05}, the decay in time in the origin $x=0$, obtaining an asymptotic behavior $\sim \beta(\sigma, m, \alpha)t^{-1/\alpha}$, with $\beta$ depending on the velocity field norm, the dimension of the Lévy process, and the parameter $\alpha$ of the Lévy measure $\nu$. This decay corresponds to the one of nonlocal PDEs. We remark that the pathwise solution profile shows no dissipativity when $\sigma$ is constant and $d=1$. It will be interesting to devote future work to the question whether pathwise dissipative behavior can be observed as well, for instance by improving the mixing property of the velocity field.

The organization of this paper is as follows. In Section \ref{sect:M-SPDE}, we introduce our model and recall the notion and basic properties of Marcus stochastic integral equations. In Section \ref{AvEnDis}, we state our theoretical results, see in particular Proposition \ref{prop}. We present and discuss numerical simulations in Section \ref{sect:num} and, finally, pose some open questions for future research in Section \ref{sect:open-questions}.

\section{Stochastic transport equations of Marcus-type}\label{sect:M-SPDE}

We consider the following transport Marcus-type SPDE on $\R_+\times \R%
^d$ 
\begin{equation}  \label{M-SPDE}
du(t,x) = (\sigma(x)\nabla u(t,x)) \diamond dZ_t,\quad u(0,x) = u_0(x),
\end{equation}
where $\sigma: \R^d \to \R^{d\times m}$ and $u_0: \R^d \to \R$. $Z$ is an $m$%
-dimensional pure jump Lévy process on a filtered probability space $%
(\Omega, \Fscr, (\Fscr_t)_{t\geq 0}, \mathbb{P})$, 
\begin{equation*}
Z_t = \int_0^t \int_{\overline{B_1(0)}}z\tilde{N}(dz,ds) + \int_0^t \int_{%
\overline{B_1(0)}^c} z N(dz,ds), \quad t \geq 0,
\end{equation*}
with Poisson random measure $N$, Lévy measure $\nu$ (i.e. $\nu
$ is a Borel probability measure on $\R^m$ with $\nu(\{0\}) = 0$ and $\int_{%
\R^m}\min(1,z^2) d\nu(z) < \infty$), and $\tilde{N}(dz,dt)= N(dz,dt) - %
\mathds{1}_{\overline{B_1(0)}}\nu(dz)dt$. Here $B_1(0)$ denotes the Euclidean ball in $\R^m$ with radius $1$ centered at $0$, and $\overline{B_1(0)}$ its closure. We make specific choices
for $\sigma$ and $\nu$ below. The symbol $\diamond$ denotes the Marcus
stochastic integral, i.e. \eqref{M-SPDE} is understood in the following
integral sense: 
\begin{align}  \label{M-SPDE-sol-eq}
u(t,x) = u_0(x) &+ \int_0^t \int_{\overline{B_1(0)}}e^{\sigma
z}u(s-,x)-u(s-,x)\,\tilde{N}(dz,ds) \\
& + \int_0^t \int_{\R^m\backslash \overline{B_1(0)}}e^{\sigma
z}u(s-,x)-u(s-,x)\,N(dz,ds)  \notag \\
& +\int_0^t \int_{\overline{B_1(0)}} e^{\sigma z}u(s-,x)-u(s-,x) - \nabla
u(s-,x)\cdot (\sigma(x)z) \,d\nu(z)ds,\quad (t,x)\in \R_+\times \R^d,  \notag
\end{align}
where for $z \in \R^m$ and $f: \R^d \to \R$, $e^{\sigma z}f$ denotes the
solution $g$ of 
\begin{align}  \label{aux-PDE}
\partial_t g(t,x) = \nabla g(t,x)\cdot (\sigma(x) z), \quad g(0,x) = f(x),
\end{align}
evaluated at $t=1$. Here $s-$ denotes the left limit of $s\in \R$. When $\sigma$ and $f$ are sufficiently regular, the
solution to this first-order linear transport PDE is unique and given by $%
g(t,x) = f(\phi_{t,0}(x))$, where $(t,x)\mapsto \phi_{t,0}(x)$ is the
inverse of the unique solution flow $(t,x)\mapsto \phi_{0,t}(x)$ for the ODE 
\begin{equation*}
\partial_t\phi_{0,t}(x) = -\sigma(\phi_{0,t}(x))z,\quad \phi_{0,0}(x) = x
\end{equation*}
on $\R\times \R^d$. We recall the following definition and result from 
\cite{HP23}.

\begin{dfn}\label{def:sol-M-SPDE}
An $(\Fscr_t)$-adapted random field $u:\R_+\times \R^d\times \Omega \to \R$
is a solution to \eqref{M-SPDE}, if it is a càdlàg $C^2$%
-semimartingale and \eqref{M-SPDE-sol-eq} is satisfied for $\mathbb{P}$-a.e. 
$\omega \in \Omega$.
\end{dfn}

\begin{prop}\label{prop:well-posedness-SPDE}
If $u_0 \in C^2_b(\R^d)$ and $\sigma\in C^4_b(\R^d,\R^{d\times m})$, then
there is a unique solution to \eqref{M-SPDE}, and it is given by  
\begin{equation*}
u(t,x) = u_0(\varphi_{t,0}(x)),
\end{equation*}
where $(t,x)\mapsto \varphi_{t,0}(x)$ denotes the inverse of the stochastic
flow of the Marcus-SDE  
\begin{equation*}
\varphi_{0,t}(x) = x - \int_0^t \sigma(\varphi_{0,s-}(x))\diamond dZ_t, \quad
t\geq 0, x \in \R^d.
\end{equation*}
\end{prop}
The definition of solution to this SDE is similar to the infinite-dimensional case, precisely it is given by
\begin{align*}
\varphi_{0,t}(x) = x &- \int_0^t \int_{\overline{B_1(0)}}\Psi_z(\varphi_{0,s-}(x)) - \varphi_{0,s-}(x)\,d\tilde{N}(dz,ds)
\\& - \int_0^t \int_{\R^m\backslash\overline{B_1(0)}}\Psi_z(\varphi_{0,s-}(x)) - \varphi_{0,s-}(x)\,dN(dz,ds)
\\& - \int_0^t \int_{\overline{B_1(0)}}\Psi_z(\varphi_{0,s-}(x)) - \varphi_{0,s-}(x) - z \sigma(\varphi_{0,s-}(x))\,\nu(z)ds,
\end{align*}
where $\Psi_z(y)$ denotes the solution to
$$\partial_t f(t) = z\sigma(f(t)), \quad t\in \R,\quad \quad f(0) = y \in \R^d,$$
evaluated at $t=1$, see \cite{M78} and \cite[Ch.4,6]{Applebaum09_Levy-proc}.

\subsection{Special cases}

We are specifically interested in the case $d=1$, $\sigma(x) = \sigma \in \R%
^m$ constant and $\nu = \frac{C}{|z|^{m+\alpha}}dz$, where $C$ is either
a constant depending on $m$ and $\alpha \in (0,2)$, or a function of $z$. In
this case, the solution to \eqref{aux-PDE} is given by $g(t,x) =
f(x+\sigma\cdot zt)$, and the last integral term in %
\eqref{M-SPDE-sol-eq} simplifies to 
\begin{equation*}
\int_0^t \int_{\overline{B_1(0)}}u(s-,x+\sigma \cdot z)-u(s-,x)-\nabla
u(s-,x)\sigma z\,d\nu(z)ds.
\end{equation*}
Moreover, in this case we have $\varphi_{0,t}(x) = x-\sigma \cdot Z_t$.
Since for any choice of $\sigma$ and $\nu$ both stochastic integrals from %
\eqref{M-SPDE-sol-eq} are martingales, taking expectation yields 
\begin{equation}  \label{exp-int}
\mathbb{E}[u(t,x)] - \mathbb{E}[u_0(x)]= \mathbb{E}\bigg[\int_0^t \int_{\overline{B_1(0)}%
}u(s-,x+\sigma \cdot z)-u(s-,x)-\nabla u(s-,x)\sigma z\,d\nu(z)ds\bigg],
\end{equation}
where we write $\mathbb{E}[X] = \int_{\Omega}X \,d\mathbb{P}$ for a random variable $X:\Omega \to \R$, provided the integral is defined.
Also note that for $d=1$, every divergence-free vector field is constant.

\section{Averaged enhanced dissipation}\label{AvEnDis}

Let $d=1$, $\sigma$ be constant, and $\nu$ have a radially symmetric density
(for instance, the classical symmetric $\alpha$-stable density $\frac{1}{|z|^{m+\alpha}}$), and set $\theta := |\sigma|$. Then, due
to the radial symmetry, the RHS of \eqref{exp-int} without expectation, i.e.
for each fixed $\omega \in \Omega$, equals 
\begin{align}  \label{int-rad-symm}
&\int_0^t \int_{\overline{B_1(0)}}u(s-,x+\theta z_1)-u(s-,x)-\nabla u(s-,x)
\theta z_1\,d\nu(z)ds,
\end{align}
where we denote by $z_i$ the $i$-th component of $z= (z_1,\dots,z_m)\in \R^m$.
In order to further calculate this integral, we need the following lemma. Below, we denote by $\pi_1$ the canonical projection $\pi_1: \R^m \to \R$, $\pi_1(z) = z_1$.

\begin{lem}\label{lem:image-measures} 
	Let $\alpha \in (0,2)$.
\begin{enumerate}
\item[(i)]  Set $\nu_\alpha := \frac{1}{|z|^{m+\alpha}}dz$. Then  
\begin{equation*}
\nu_{\alpha,1}: = \nu\circ \pi_1^{-1} = \frac{C(m,\alpha)}{|y|^{1+\alpha}}dy,
\end{equation*}
with
$$C(m,\alpha) = |\mathbb{S}^{m-2}|\int_0^\infty(1+r^2)^{-\frac{m+\alpha%
}{2}}r^{m-2}dr< \infty,$$
 where $|\mathbb{S}^{m-2}|$ denotes the surface area of $\mathbb{S}^{m-2}$, the unit sphere in $\R^{m-1}$.

\item[(ii)] Let $\nu_{\mathds{1},\alpha} := \mathds{1}_{\overline{B_1(0)}}(z)\frac{1}{%
|z|^{m+\alpha}}dz$. Then 
\begin{equation*}
\nu_{\mathds{1},\alpha,1} := \nu_{\mathds{1},\alpha}\circ \pi_1^{-1} = \frac{C(y,m,\alpha)%
}{|y|^{1+\alpha}}dy,
\end{equation*}
with
\begin{equation*}
C(y,m,\alpha) := \mathds{1}_{[-1,1]}(y) |\mathbb{S}^{m-2}|\int_0^{\frac{\sqrt{%
1-y^2}}{|y|}}(1+r^2)^{-\frac{m+\alpha}{2}}r^{m-2}dr,\quad y \in \R.
\end{equation*}
\end{enumerate}
\end{lem}

\begin{proof}
	\begin{enumerate}
		\item [(i)]
	Let $A\in \Bscr(\R)$, and for $z=(z_1,\dots,z_m) \in \R^m$, write $z' = (z_2,\dots,z_m)$. Then
\begin{align*}
	\nu_{\alpha,1}(A) = \nu_\alpha(A\times \R^{m-1}) &= \int_A\int_{\R^{m-1}}(z_1^2+|z'|^2)^{-\frac{m+\alpha}{2}}dz_1dz' 
	\\&= \int_A   |z_1|^{-m-\alpha}\int_{\R^{m-1}}\bigg(1+\frac{|z'|^2}{z_1^2}\bigg)^{-\frac{m+\alpha}{2}}dz' dz_1
	\\&= \int_A |z_1|^{-1-\alpha} \int_{\R^{m-1}}(1+|z'|)^{-\frac{m+\alpha}{2}}dz'dz_1
	\\&= \int_A  C(m,\alpha)|z_1|^{-1-\alpha}dz_1,
\end{align*}
where the third equality follows from the transformation rule and the final one by calculating the inner integral by spherical coordinates.
\item[(ii)] Due to the definition of $\nu_{\mathds{1},\alpha}$, the proof is similar to the first part. 
	\end{enumerate}
\end{proof}
To convey our further procedure, first
consider \eqref{int-rad-symm}  with domain of integration  $\R^m$ instead of $\overline{B_1(0)}$, and choose $\nu = C(m,\alpha)^{-1}\nu_\alpha$. Then, by
Lemma \ref{lem:image-measures} (i), \eqref{int-rad-symm} equals 
\begin{equation*}
\int_0^t \int_\R \frac{u(s-,x+\theta y)-u(s-,x)-\nabla u(s-,x)\theta y}{%
|y|^{1+\alpha}}\,dy ds.
\end{equation*}
Inserting in \eqref{exp-int} and interchanging the expectation with the
temporal and spatial integral as well as with the gradient shows that $%
U(t,x):= \mathbb{E}[u(t,x)]$ solves 
\begin{equation}  \label{eq-U}
\partial_tU(t,x) = \mathcal{L}_{\alpha}U(t-,x),
\end{equation}
where the operator $\mathcal{L}_{\alpha}$ is defined by
\begin{equation*}
\mathcal{L}_{\alpha}f(x) := \int_{\R} \frac{f(x+\theta y)-f(x)-\nabla
f(x)\theta y}{|y|^{1+\alpha}}dy.
\end{equation*}
Note that \eqref{eq-U} is a deterministic, nonlocal second-order parabolic equation for
the expected value of $u(t,x)$, which itself solves (pathwise) \eqref{M-SPDE}. We point out the similarity of $\mathcal{L}%
_{\alpha}$ with the fractional
Laplacian $(-\Delta)^\alpha$ on $\R$,
\begin{equation*}
(-\Delta)^\alpha f(x) = \int_{\R} \frac{f(x+y)-f(x)-\nabla f(x)y \mathds{1}%
_{(-1,1)}(y)}{|y|^{1+\alpha}}dy.
\end{equation*}
In fact, for $\theta = 1$ and $\alpha >1$, $\mathcal{L}_\alpha = (-\Delta)^\alpha$, since in this case $\int_{B_1(0)^c}\frac{y}{|y|^{1+\alpha}}dy = 0$.

Now, in order to take into account the proper domain of integration $\overline{B_1(0)}$ in \eqref{int-rad-symm}, we repeat the previous lines with the choice 
\begin{equation}\label{choice-nu}
\nu
= |\mathbb{S}^{m-2}|^{-1}\nu_{\mathds{1},\alpha}.
\end{equation}
 Then, similarly to \eqref{eq-U}, we arrive at 
\begin{equation}\label{eq:det-eq}
\partial_t U(t,x) = \mathcal{L}_{\mathds{1},\alpha}U(t-,x),\quad (t,x)\in \R_+\times \R,
\end{equation}
satisfied by $U(t,x) = \mathbb{E}[u(t,x)]$, where we set 
\begin{equation*}
\mathcal{L}_{\mathds{1},\alpha}f(x) := \int_{-1}^1 f(x+\theta y)-f(x)-\nabla
f(x)\theta y \,d\nu_{\mathds{1},\alpha,1}(y) = \int_{-1}^1c(m,\alpha,y)\frac{%
f(x+\theta y)-f(x)-\nabla f(x)\theta y}{|y|^{1+\alpha}}\,dy 
\end{equation*}
with 
$$c(y,m,\alpha):= \int_0^{\frac{\sqrt{1-y^2}}{|y|}}(1+r^2)^{-\frac{%
m+\alpha}{2}}r^{m-2}dr.$$ For the finite positive weight $c(y,m,\alpha)$, we
note $c(y,m,\alpha)\xrightarrow{|y|\to 1}0$, and that $c(y,m,\alpha)$ is symmetric around $y=0$. Since $c(y,m,\alpha)$ is bounded on $(-1,1)$, by Taylor formula $\mathcal{L}_{\mathds{1},\alpha}f(x)$ is well-defined and finite for any $\alpha \in (0,2)$, $x \in \R$, and $f \in C^2(\R)$. The latter is satisfied for $x\mapsto u(t-,x)$, for every $t \geq 0$ and $\mathbb{P}$-a.e. $\omega \in \Omega$, as well as for $x\mapsto U(t-,x)$, for every $t \geq 0$. Interchanging the expectation  with the temporal and spatial integral in \eqref{int-rad-symm} is justified, since $u$ is given as in Proposition \eqref{prop:well-posedness-SPDE} and since $\varphi$ is a stochastic flow of smooth diffeomorphisms, see \cite[Thm.6.10.10]{Applebaum09_Levy-proc}. Therefore, we have arrived at the following result.

\begin{prop}\label{prop}
	Consider \eqref{M-SPDE} for $d=1$, $u_0 \in C^2_b(\R)$, $\sigma \in \R^m$ constant, $Z$ with Lévy measure $\nu_{\mathds{1},\alpha}$ as in \ref{choice-nu}, and let $u$ be the unique solution in the sense of Definition \ref{def:sol-M-SPDE}. Then, $U(t,x):= \mathbb{E}[u(t,x)]$ solves \eqref{eq:det-eq}.
\end{prop}

\begin{rem}\label{decaylap}
Note that equation \eqref{eq:det-eq} is dissipative since it involves a fractional Laplace-like operator. For such operators, there are results for the decay of solutions, yielding a link with our numerical results (see Section \ref{sect:num}) and the operator $\mathcal{L}_{\mathds{1},\alpha}$.
More precisely, consider
$$
\partial_t u(t,x)=(-\Delta)^{\alpha/2}u(t,x),\quad u(0,x)=u_0(x),\quad (t,x) \in \R_+\times \R.
$$
There exists a $C_0-$semigroup $(\mathcal{S}_\alpha(t))_{t\geq 0}$ such that $u(t,x)=\mathcal{S}_\alpha(t)*u_0(x)$. The following estimate in dimension $d=1$ was obtained in \cite{FK20}.
$$
\|u(t)\|_\infty=\|\mathcal{S}_\alpha(t)* u_0\|_{\infty}\leq t^{-\alpha/2}\|u_0\|_1,
$$
where $\|\cdot\|_\infty$ and $||\cdot||_1$ denote the usual $L^\infty$- and $L^1$-norms, respectively, thus showing the dissipative behavior of the solution. This behavior is retrieved also numerically (see below), suggesting a similar behavior for $\mathcal{L}_{\mathds{1},\alpha}$ due to is similarity with the fractional Laplacian.

So far, we were unable to prove similar results for individual paths of the solution to \eqref{M-SPDE}. In fact, the pathwise profile shows no sign of dissipativity (see Figure \ref{pwsol_s05}) in our simple special case, where $\sigma$ is constant and $d=1$. So, a conclusion on the pathwise dissipative behavior and an Ito-Stratonovich diffusion limit-type result is yet to be reached.
A main reason for this appears to be the absence of any mixing property of divergence-free vector fields in dimension $d=1$, yielding the pathwise profile a simple translation in time. We expect that in dimension $d>1$ and for suitable vector fields as in \cite{FGL22,FGL21}, a pathwise dissipative behavior of the solution to \eqref{M-SPDE} solution is possible.
\end{rem}

\section{Numerical results}\label{sect:num}
\begin{figure}[b!]
	\includegraphics[width=\textwidth]{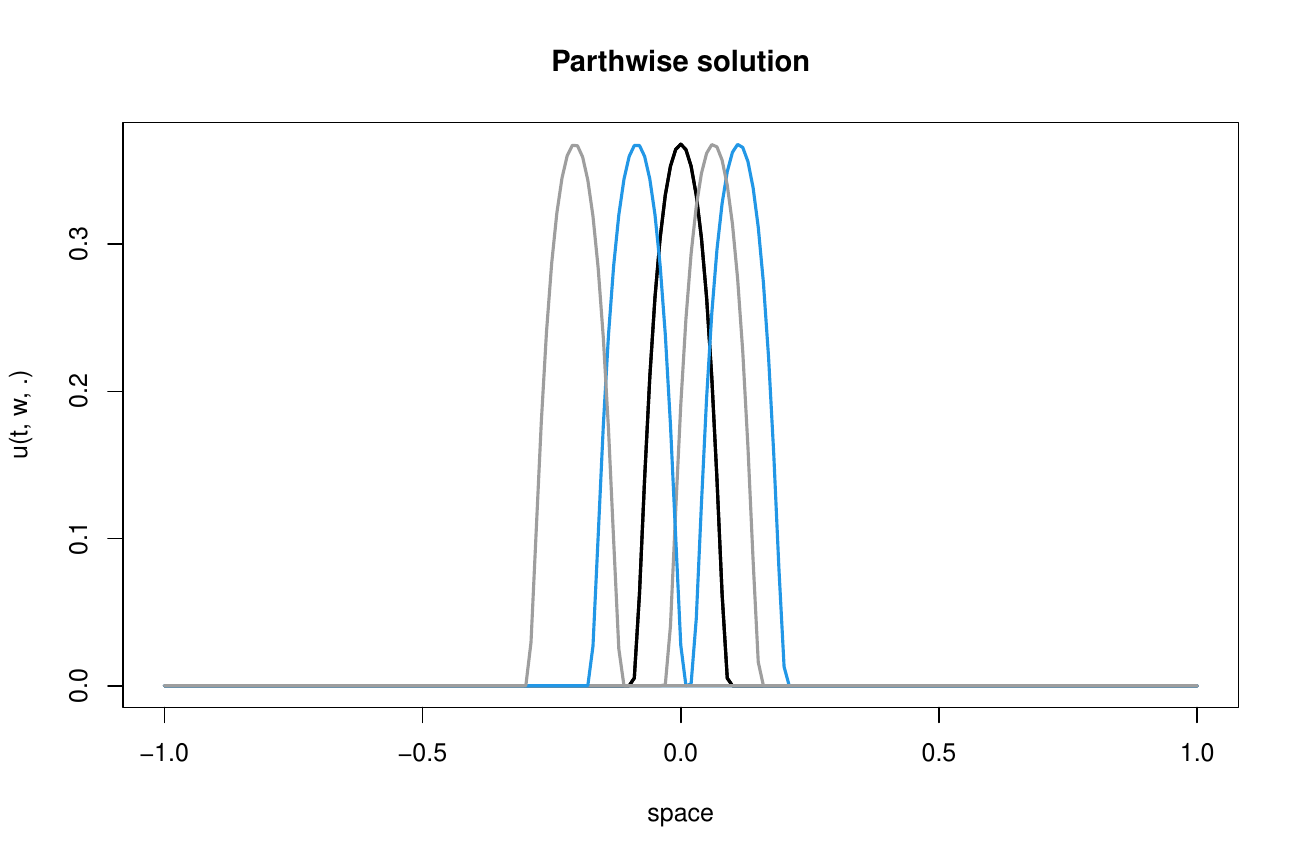}
	\caption{Solution trajectory at several times}
	\label{pwsol_s05}
\end{figure}

In order to support the results of the previous section, here we present numerical simulations of the SPDE \eqref{M-SPDE} in our special case, thereby underlining the dissipative character of the expected value of its solution.
As in Section \ref{AvEnDis}, we limit ourselves to dimension $d=1$, we consider a constant vector field $\sigma\in \mathbb{R}^m$, and we choose the Lévy measure of the driving pure jump Lévy process to be $\nu_{\mathds{1},\alpha}$, as defined in Lemma \ref{lem:image-measures} (ii). The jump dimension $m\in \mathbb{N}$ of the Lévy process is arbitrary.
Under these assumptions, an explicit solution to \eqref{M-SPDE} can be computed via the method of characteristics, which can be exploited for simple numerical simulations, namely
\begin{align*}
	u(t,x;\omega):=u_0(x+\sigma\cdot Z_t(\omega)).
\end{align*}
For all simulations below, we fixed a smooth bump function
$u_0(x):=\exp(-\frac{0.01}{0.01-\min(0.01,x^2)})$ as initial condition.
To simulate our $\alpha$-stable Lévy process trajectory, we need to take into account the fact that our choice of Lévy measure $\nu_{\mathds{1},\alpha}$ neglects large jumps. To this end, using the independent increment and self-similarity property, we compute the next step of the trajectory by cutting away jumps larger than one and generating a new realization until the jump size is sufficiently small. To implement the $\alpha-$stable distribution we used the $R$-package $Stabledif$ \cite{Stabledist}.
The time step is $dt=10^{-4}$, in the range $[0,2]$, while the space domain is selected to be the interval $[-1,1]$. Note that the equation \ref{M-SPDE} is posed on $\R^d$, i.e. there is a slight discrepancy between the equation and our numerical simulations
 that needs to be taken into account when we interpreting the results. The support of $\supp u_0$ is sufficiently small so that almost no mass escapes the system throughout the time evolution. The space discretization size is selected at $dx=10^{-3}$.

\begin{figure}[h!]
	\centering
	\begin{subfigure}[t]{0.7\textwidth}
		\centering
		\includegraphics[width=\textwidth]{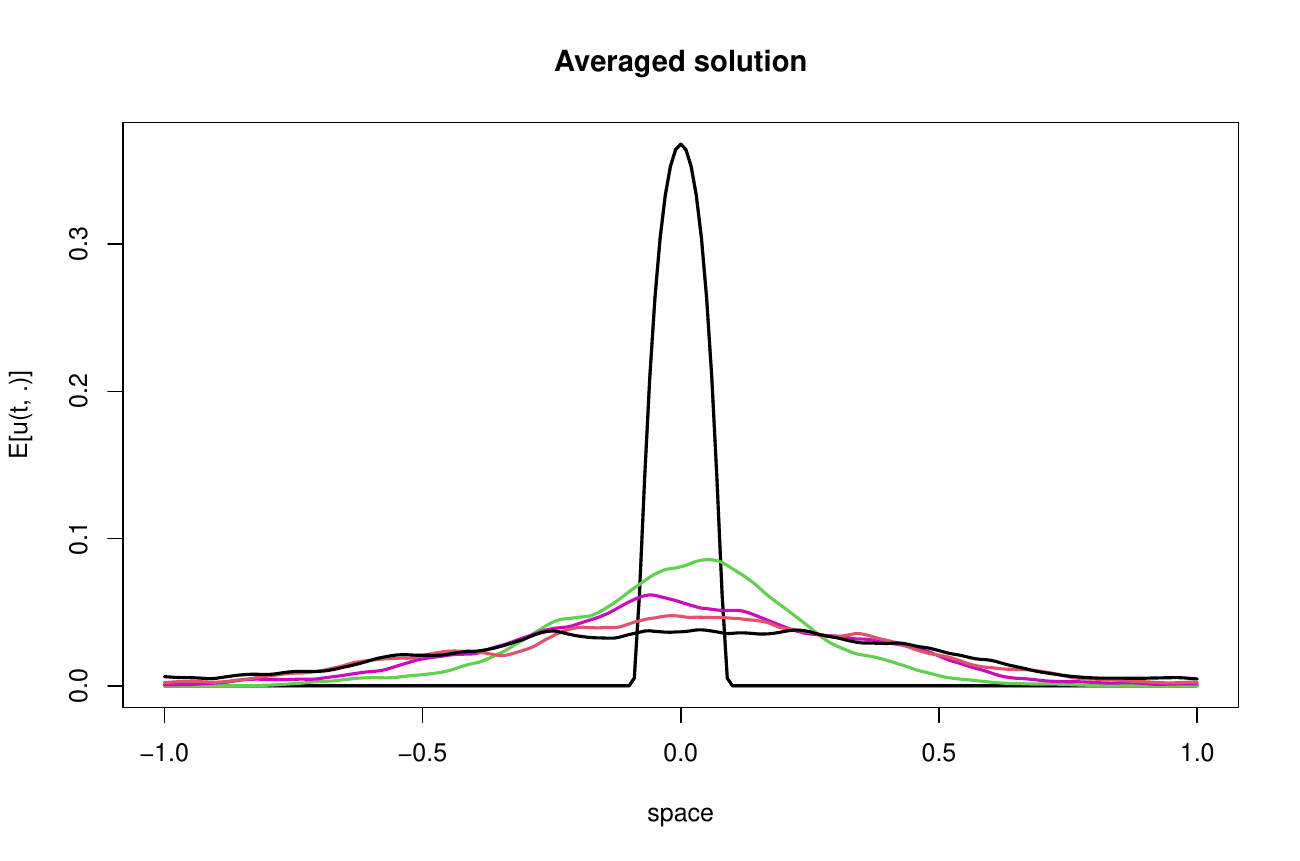}
		
		\caption{Decay in space of averaged solution at several times}\label{awsol_s05}
	\end{subfigure}%
	\hfill
	\begin{subfigure}[t]{0.7\textwidth}
		\centering
		\includegraphics[width=\textwidth]{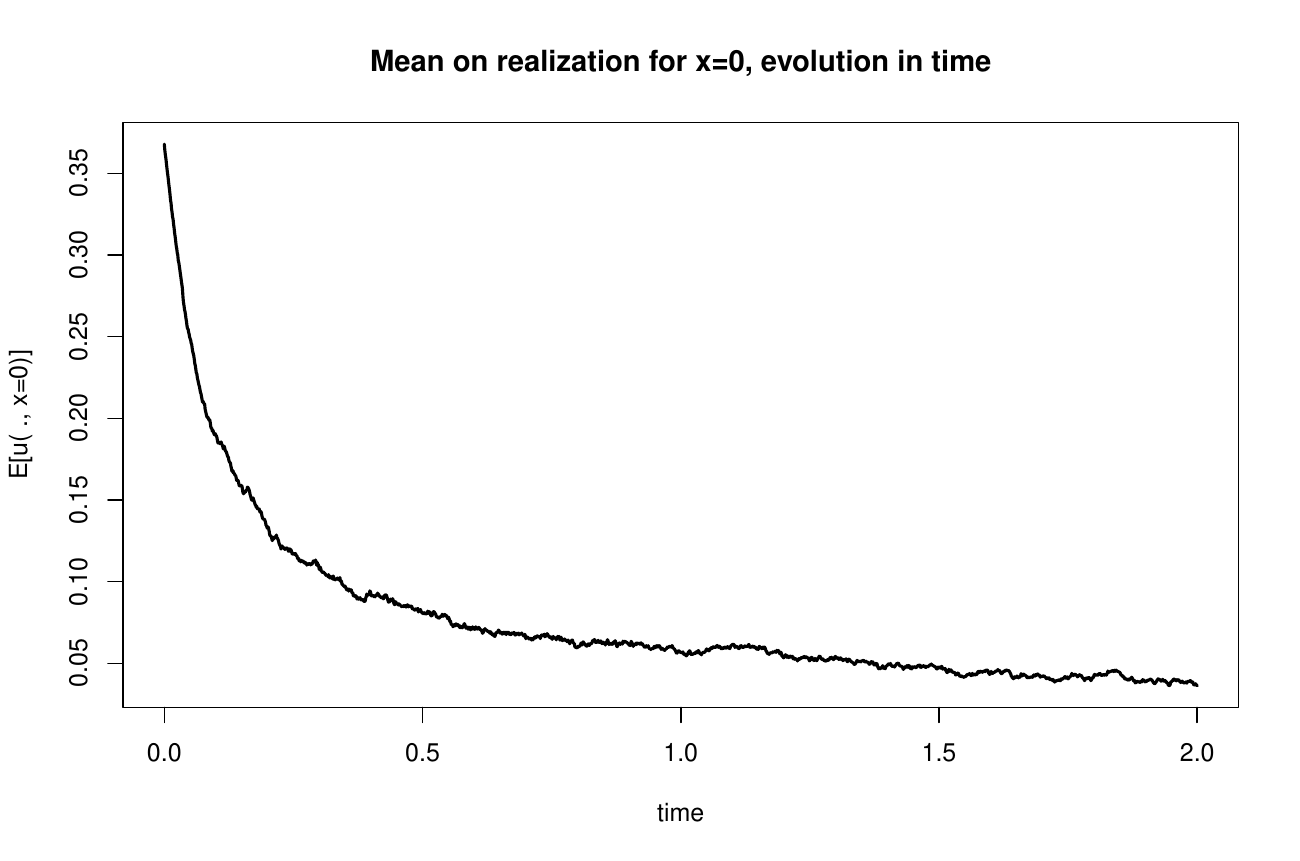}
		\caption{Decay in time of averaged solution at $x=0$}
		\label{da_s05}
	\end{subfigure}
	\caption{Averaged solution to equation \ref{M-SPDE}}
\end{figure}

In our numerical analysis, we set our parameters in the range $\theta=\|\sigma\|\in[0,1]$, $\alpha\in (0,2)\setminus \{1\}$ and $m=1,2,5,10$. Our results are qualitatively consistent, and we are here analyzing and presenting figures for $\theta=0.5,\ \alpha=1.5$ and $m=2$. In future works, we expect to give a precise quantification of the dependence on these parameters or the decay rate in time of the averaged solution. Our results, as of now, focus on the qualitative behavior in time for the averaged solution. This behavior depends only on the parameter $\alpha$ of the stable distribution.

\begin{figure}[h!]
	\centering
	\begin{subfigure}[b]{0.7\textwidth}
		\centering
		\includegraphics[width=\textwidth]{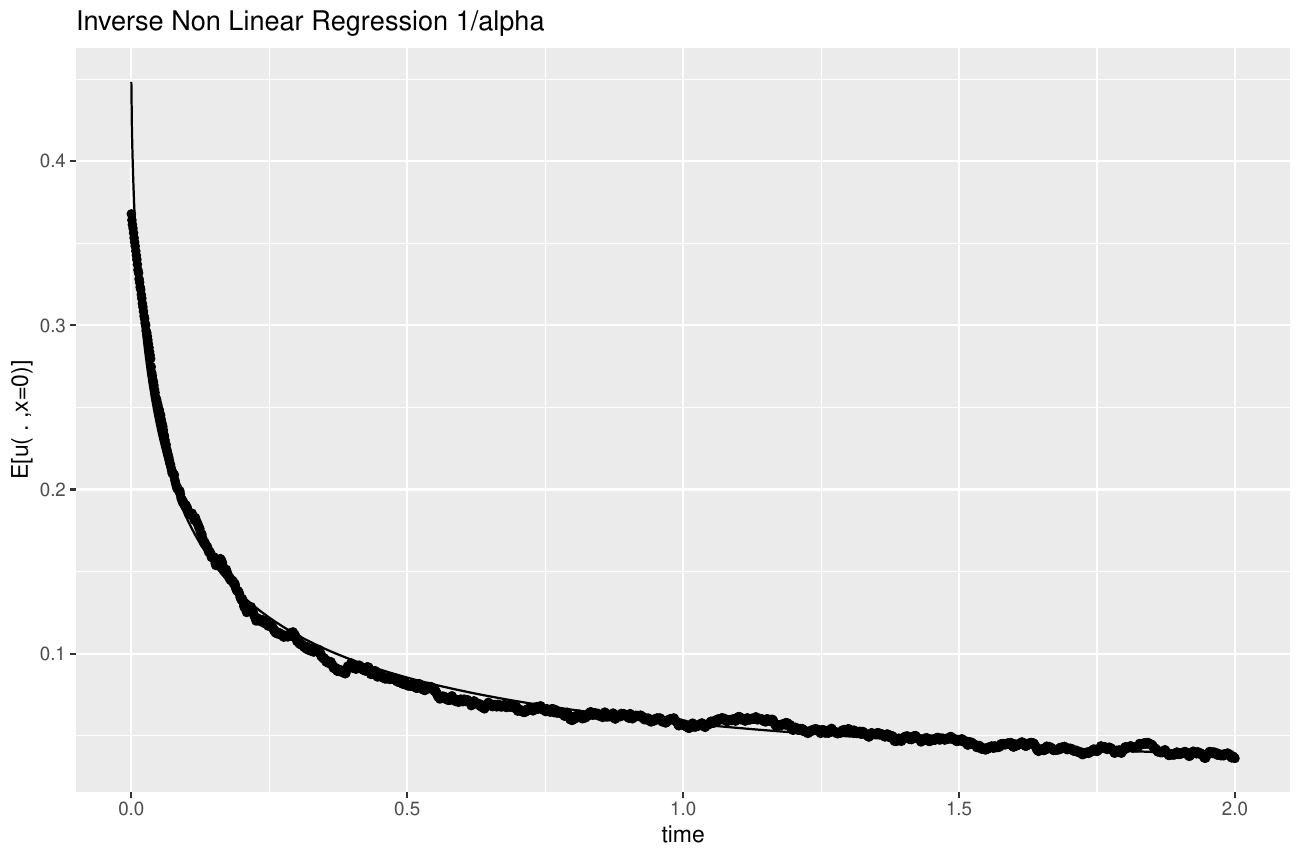}
		
		\caption{Nonlinear regression $\sim t^{-1/alpha}$, $R.E.=0.003$}\label{regsol_s05}
	\end{subfigure}%
	\hfill
	\begin{subfigure}[b]{0.7\textwidth}
		\centering
		\includegraphics[width=\textwidth]{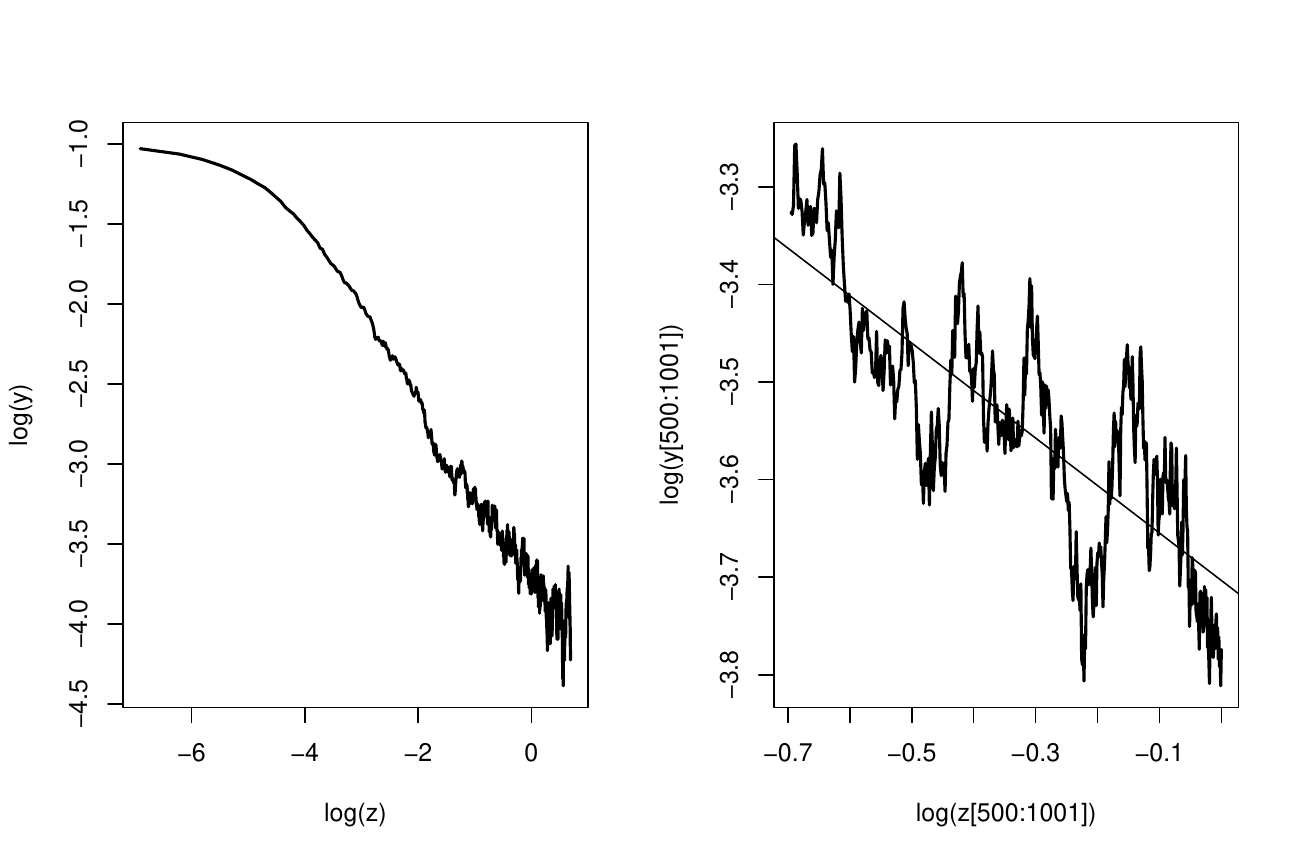}
		\caption{Log-Log Plot and tail behavior. $z:=t\ (time),\ y:=\mathbb{E}[u(t,x=0)]$.\\ Left: full decay. Right: zoom on tail for larger times.}
		\label{log_s05}
	\end{subfigure}
	\caption{Nonlinear regression on \ref{eq:det-eq}}
\end{figure}

Figure \ref{pwsol_s05} presents a solution trajectory to \eqref{M-SPDE} at different times. Precisely, the black curve represents the initial condition, while the blue ones show the solution profile at times $t=1/2,\ t=1$. Similarly, the grey curves represent the times $t=3/2,\ t=2$. As expected, the solution trajectories are translations, with no mixing property arising from the constant vector field $\sigma$, thus not showing any dissipation. The $L^2$ norm in space is preserved, and even for large times the solution preserves energy with an error of $10^{-14}$, which is only due to the space domain being finite in the simulations.
Nonetheless it is clear that the energy is preserved, as expected, due to the transport character of \eqref{M-SPDE}.

To show consistency with the theoretical results, we simulate the evolution of the averaged solution profile, obtained with Monte Carlo methods over averaging $5000$ samples, as pictured in Figures \ref{awsol_s05}, \ref{da_s05}, showing a decay in both time and space. In the first figure, analogously to the pathwise result, we present the averaged solution at different times. The black graph represents the initial condition, while the green and purple ones show the average solution at times $t=1/2,\ t=1$, respectively, and at times $t=3/2,\ t=2$ for the red and blue ones.
In this case, a diffusive behavior is present and a decay in space and time is observed. Concerning the space behavior, the profiles are still not smooth. The reason is twofold, one being the averaging procedure, the second one arising from the fractional operator obtained for the equation modeling the averaged system (i.e. \eqref{eq:det-eq}) and the corresponding small jumps. Concerning the decay, particularly care for $x=0$, in which the initial condition has its maximum, and note the decay of the averaged solution in time with a power law-like asymptotic behavior.

More so, starting with a compact supported initial condition, we show, in Figures \ref{regsol_s05} and \ref{log_s05}, the time decay in the origin $x=0$, with a nonlinear regression to estimate the asymptotic and power law-like behavior, linking it to the operator $\mathcal{L}_{\mathds{1},\alpha}$ proposed in the theoretical section.
In particular, in Figure \ref{regsol_s05}, we performed a regression using the insight of Remark \ref{decaylap} numerically to show that, with a residual error of less than $0.003$, we have a decay of the averaged solution in time of the following form:
$$
\mathbb{E}[u(t,x=0)]\sim \beta(\sigma, m, \alpha)t^{-1/\alpha},
$$
with $\beta$ depending on $\theta:=\|\sigma\|$, the dimension of the Lévy process $m\in\mathbb{N}$, and the parameter $\alpha$ of the Lévy measure $\nu$. 
Here, we have not delved into an analysis of the behavior of $\beta$, which could be theoretically examined, as discussed in Remark \ref{decaylap}. Therefore, it is crucial for future research to explore how the strength of $\beta$, and consequently the velocity field, interact with the decay of the profile.

In Figure \ref{log_s05}, we plot the log-log version of the decay in time, showing the inverse asymptotic behavior of the averaged solution profile in space and time and its concordance with the results of the theoretical section and the nonlinear regression. More so, on the right in Figure \ref{log_s05}, the tail was analyzed in the time frame $z[500:1001]:=[0.5,1]$, showing a rough behavior, but with respect to the regression line, the error is in the range $0.003-0.006$, which validates our results.

\section{Open questions}\label{sect:open-questions}
We conclude this work with a few questions for future work in this direction.
\begin{enumerate}
	\item [(i)] Can similar results be obtained for non-constant $\sigma$ (in dimension $d \geq 2$)?
	\item[(ii)] 
	Related to (i), do vector fields with suitable mixing properties lead to a  pathwise dissipation result for solutions to \eqref{M-SPDE}?
	\item[(iii)] Is it possible to obtain precisely the fractional Laplacian as the operator of the deterministic equation \eqref{eq:det-eq} instead of $\mathcal{L}_{\mathds{1},\alpha}$, and what is the corresponding Lévy measure?
	\item[(iv)] Which operators does one obtain by choosing Lévy measures in \eqref{M-SPDE} which are not $\alpha$-stable?
\end{enumerate}

\paragraph{Acknowledgements.} The research of F.F. is funded by the European Union (ERC, NoisyFluid, No. 101053472. 
A.P. is supported by the European Union (ERC, StochMan, No. 101088589).
M.R. is funded by the German Research Foundation (DFG) - Project number 517982119. The authors would also like to thank Marvin Weidner for a very valuable comment regarding Lemma \ref{lem:image-measures} and Gaia Tramonte for helpful suggestions.


\bibliographystyle{plain}
\bibliography{bib-collection}

\end{document}